\documentclass[11pt,letterpaper]{amsart}
\usepackage{hyperref}
\usepackage{graphicx, color}
\usepackage{mathtools}
\usepackage{enumerate}
\usepackage[margin=1in]{geometry}

\newtheorem{theorem}{Theorem}[section]

\newtheorem{lemma}[theorem]{Lemma}
\newtheorem{corollary}[theorem]{Corollary}
\theoremstyle{definition}

\theoremstyle{remark}

\DeclareMathOperator{\sign}{sgn}
\DeclareMathOperator{\arccosh}{arccosh}
\DeclareMathOperator{\Arg}{Arg}

\renewcommand{\Re}{\operatorname{Re}}
\renewcommand{\Im}{\operatorname{Im}}

\begin{document}
\title{An elementary and unified proof of Grothendieck's inequality}
\author{Shmuel~Friedland}
\address{Department of Mathematics, Statistics and Computer Science,  University of Illinois, Chicago, IL 60607-7045}
\email{friedlan@uic.edu}
\author{Lek-Heng~Lim}
\address{Computational and Applied Mathematics Initiative, Department of Statistics,
University of Chicago, Chicago, IL 60637-1514}
\email[corresponding author]{lekheng@galton.uchicago.edu}
\author{Jinjie~Zhang}
\address{Department of Statistics, University of Chicago, Chicago, IL 60637-1514}
\curraddr{Department of Mathematics, University of California, San Diego, CA 92093-0112}
\email{jiz003@ucsd.edu}

\begin{abstract}
We present an elementary, self-contained proof of Grothendieck's inequality that unifies the real and complex cases and yields both the Krivine and Haagerup bounds, the current best-known explicit bounds for the real and complex Grothendieck constants respectively.  This article is intended to be pedagogical, combining and streamlining known ideas of Lindenstrauss--Pe{\l}czy\'nski, Krivine, and Haagerup into a proof that need only univariate calculus, basic complex variables, and a modicum of linear algebra as prerequisites.
\end{abstract}


\maketitle

\section{Introduction}\label{sec:intro}

We will let $\mathbb{F} = \mathbb{R}$ or $\mathbb{C}$ throughout this article. In $1953$, Grothendieck proved a powerful result that he called ``the fundamental theorem in the metric theory of tensor products''  \cite{Grothendieck}; he showed that there exists a finite  constant $K > 0$ such that for every $l,m,n\in\mathbb{N}$ and every matrix $M=(M_{ij})\in\mathbb{F}^{m\times n}$,
\begin{equation}\label{GI}
\max_{\lVert x_i\rVert = \lVert y_j \rVert = 1}\biggl| \sum_{i=1}^m\sum_{j=1}^n M_{ij} \langle x_{i}, y_{j}\rangle\biggr|\leq K \max_{\lvert \varepsilon_i \rvert = \lvert \delta_j \rvert = 1}\biggl|\sum_{i=1}^m\sum_{j=1}^n M_{ij} \varepsilon_i \delta_j\biggr|
\end{equation}
where $\langle \cdot,\cdot \rangle$ is the standard inner product in $\mathbb{F}^l$, the maximum on the left is taken over all $x_i,y_j\in\mathbb{F}^{l}$ of unit $2$-norm,  and the maximum on the right is taken over all  $\varepsilon_i, \delta_j \in \mathbb{F}$ of unit absolute value (i.e., $\varepsilon_i = \pm 1$,  $\delta_j=\pm1$ over $\mathbb{R}$; $\varepsilon_i = e^{i\theta_i}$, $\delta_j= e^{i \phi_j}$ over $\mathbb{C}$). The inequality \eqref{GI} has since been christened \emph{Grothendieck's inequality} and the smallest possible constant $K$ \emph{Grothendieck's constant}. The value of Grothendieck's constant depends on the choice of $\mathbb{F}$ and we will denote it by $K_G^\mathbb{F}$.

Over the last 65 years, there have been many attempts to improve and simplify the proof of Grothendieck's inequality, and also to obtain better bounds for the Grothendieck constant $K_G^\mathbb{F}$, whose exact value remains unknown. The following are some major milestones:
\begin{enumerate}[\upshape (i)]
\item The central result of Grothendieck's original paper \cite{Grothendieck} is that his eponymous inequality holds with $\pi/2\leq K_G^\mathbb{R}\leq \sinh (\pi/2) \approx 2.301$ and $1.273 \approx 4/\pi\leq K_G^\mathbb{C}$. Grothendieck relied on the sign function for the real case and obtained the complex case from the real case via a complexification argument.

\item The power of Grothendieck's inequality was not generally recognized until the work of Lindenstrauss and Pe{\l}czy\'nski \cite{Lindenstrauss} 15 years later, which connected the inequality to absolutely $p$-summing operators.  They elucidated and improved Grothendieck's proof in the real case by computing expectations of sign functions and using Taylor expansions, although they did not get better bounds for $K_G^\mathbb{R}$.

\item Rietz \cite{Rietz} obtained a slightly smaller bound $K_G^\mathbb{R}\leq2.261$ in 1974 by averaging over $\mathbb{R}^n$ with normalized Gaussian measure and using a variational argument to determine an optimal scalar map corresponding to the sign function.

\item Our current best known upper bounds for $K_G^\mathbb{R}$ and $K_G^\mathbb{C}$ are due to Krivine \cite{Krivine2}, who in 1979 used Banach space theory  and ideas in \cite{Lindenstrauss} to get 
\[
K_G^\mathbb{R}\leq \frac{\pi}{2\log(1+\sqrt{2})}\approx 1.78221;
\]
and Haagerup \cite{Haagerup}, who in 1987 extended Krivine's ideas to $\mathbb{C}$ to get
\[
K_G^\mathbb{C}\leq  \frac{8}{\pi(x_0+1)} \approx 1.40491,
\]
where $x_0 \in [0,1]$ is the unique solution to:
\[
x\int_0^{\pi/2}\frac {\cos^2t}{\sqrt{1-x^2\sin^2t}}\,dt=\frac \pi 8 (x+1).
\]

\item Our current best known lower bounds for $K_G^\mathbb{R}$ and $K_G^\mathbb{C}$ are due to  Davie \cite{Davie,Davie2}, who in 1984 used spherical integrals to get
\begin{gather*}
K_G^\mathbb{R} \geq   \sup_{x\in(0,1)} \frac{1-\rho(x)}{\max(\rho(x),f(x))} \approx 1.67696,
\shortintertext{where}
\rho(x) \coloneqq \sqrt{\frac 2 \pi}xe^{-x^2/2}, \quad
f(x) \coloneqq \frac 2 \pi e^{-x^2}+\rho(x)\biggl[1-\sqrt{\frac{8}{\pi}}\int_x^\infty e^{-t^2/2}\, dt \biggr];
\shortintertext{and}
K_G^\mathbb{C}\geq  \sup_{x>0}\frac{1-\theta(x)}{g(x)} \approx 1.33807,
\shortintertext{where}
\begin{aligned}
\theta(x)&\coloneqq \frac 1 2 \biggl[1-e^{-x^2}+x\int_x^\infty e^{-t^2}\, dt \biggr],\\
g(x)&\coloneqq \biggl[ \frac 1 x (1-e^{-x^2})+\int_x^\infty e^{-t^2}\, dt  \biggr]^2+\theta(x)\biggl[ 1-\frac 2 x (1-e^{-x^2})\biggr].
\end{aligned}
\end{gather*}

\item Progress on improving the aforementioned bounds halted for many years. Believing that Krivine's bound is the exact value of $K_G^\mathbb{R}$, some were  spurred to find matrices that yield it as the lower bound of $K_G^\mathbb{R}$ \cite{Konig}. The belief was dispelled in $2011$ in a landmark paper \cite{Braverman}, which demonstrated the existence of a positive constant $\varepsilon$ such that $K_G^\mathbb{R}<\pi/\bigl(2\log(1+\sqrt{2})\bigr)-\varepsilon$ but the authors did not provide an explicit better bound. To date,  Krivine's  and Haagerup's bounds remain the best known explicit upper bounds for $K_G^\mathbb{R}$  and $K_G^\mathbb{C}$ respectively.

\item  There have also been many alternate proofs of Grothendieck's inequality employing a variety of techniques, among them 
factorization of Hilbert spaces \cite{Maurey, Jameson,Pisier},
absolutely summing operators \cite{Diestel2, Lindenstrauss,Pisier2},
geometry of Banach spaces \cite {Albiac, Lindenstrauss2},
metric theory of tensor product \cite{Diestel},
basic probability theory \cite{Blei},
bilinear forms on $C^{*}$-algebra \cite{Kaij}.
\end{enumerate}

In this article, we will present a proof of Grothendieck's inequality that unifies both the (a) real and
(b) complex cases; and yields both the (c) Krivine and (d) Haagerup
bounds \cite{Krivine2, Haagerup}. It is also elementary in that it requires
little more than standard college mathematics. Our proof will rely on Lemma~\ref{2-thm1}, which
is a variation of known ideas in \cite{Lindenstrauss, Haagerup, Jameson}.
In particular, the idea of using the sign function to establish \eqref{GI}
in the real case was due to Grothendieck himself \cite{Grothendieck} and
later also appeared in \cite{Lindenstrauss, Krivine2}; whereas the use of the
sign function in the complex case first appeared in \cite{Haagerup}. To be
clear, all the key ideas in our proof were originally due to
Lindenstrauss--Pe{\l}czy\'nski, Krivine, and Haagerup \cite{Lindenstrauss,
Krivine2, Haagerup}, our only contribution is pedagogical --- combining,
simplifying, and streamlining their ideas into what we feel is a more
palatable proof. To understand the proof, readers need only know univariate calculus, basic complex variables, and a small amount of linear algebra. We will use some basic Hilbert space theory and tensor product constructions in Section~\ref{sec:proof} but both notions will be explained in a self-contained and elementary way.

\section{Gaussian integral of sign function}\label{sec:gauss}

Throughout this article, our inner product over $\mathbb{C}$ will be sesquilinear in the second argument, i.e.,
\[
\langle x,y\rangle := y^* x \qquad \text{for all}\; x, y \in \mathbb{C}^n.
\]

For $z\in\mathbb{F}= \mathbb{R}$ or $\mathbb{C}$, the sign function is
\begin{equation}\label{eq:sgn}
\sign(z)=\begin{cases}
z/|z| & \text{if}\; z\neq0, \\
0 & \text{if}\; z=0;
\end{cases}
\end{equation}
and for $z \in \mathbb{F}^n$, the Gaussian function is
\[
G_n^\mathbb{F}(z)=\begin{cases}
(2\pi)^{-n/2}\exp\bigl(- \|z\|_{2}^{2}/2\bigr) &\text{if}\; \mathbb{F}=\mathbb{R},\\
\pi^{-n} \exp(-\|z\|_2^2) &\text{if}\; \mathbb{F}=\mathbb{C}.
\end{cases}
\]
Lemma~\ref{2-thm1} below is based on \cite{Jameson, Haagerup}; the complex version in particular is a slight variation of \cite[Lemma~3.2]{Haagerup}. It plays an important role in our proof because the right side of \eqref{2-eq1}  depends only on the inner product $\langle u,v\rangle$ and not (explicitly) on the dimension $n$. In addition, the functions on the right are homeomorphisms and admit Taylor expansions, making it possible to expand them in powers $\langle u,v\rangle^d$, which will come in useful when we prove Theorem~\ref{2-thm2}.
\begin{lemma}\label{2-thm1}
Let $u,v \in \mathbb{F}^{n}$ with $\|u\|_2=\|v\|_2=1$. Then
\begin{equation}\label{2-eq1}
\int_{\mathbb{F}^n}  \sign\langle u,z\rangle \sign\langle z,v \rangle G_n^\mathbb{F}(z)\,dz=\begin{cases}
\dfrac{2}{\pi}\arcsin\langle u,v\rangle &\text{if} \; \mathbb{F}=\mathbb{R},\\
\displaystyle\langle u,v\rangle \int_0^{\pi / 2}
\frac{\cos^2t}{(1-|\langle u,v\rangle|^2\sin^2t)^{1/2}}\,dt &\text{if} \; \mathbb{F}=\mathbb{C}.
\end{cases}
\end{equation}
\end{lemma}
\begin{proof}
\textsc{Case} I: $\mathbb{F}=\mathbb{R}$.\; Let $\arccos\langle u,v\rangle=\theta$, so that $\theta\in[0,\pi]$ and $\arcsin\langle u,v\rangle=\pi / 2-\theta$.
Choose $\alpha,\beta$ such that $0<\beta-\alpha<\pi$ and define
\[
E(\alpha,\beta)=\{(r\cos\theta,r\sin\theta,x_{3},\dots,x_{n}):r\geq 0, \alpha\leq\theta\leq\beta\}.
\]
The Gaussian measure of a measurable set $A$ is the integral of  $G_n^\mathbb{R}(x)$ over $A$. Upon integrating with respect to $x_{3},\dots,x_{n}$, the following term remains:
\[
\frac {1}{2\pi}\int_{E(\alpha,\beta)} e^{-\frac {1}{2} (x_{1}^{2}+x_{2}^{2}) }\,dx_{1}\,dx_{2}=\frac {1}{2\pi}\int_{\alpha}^{\beta}d\theta\int_{0}^{\infty}re^{-\frac 1 2 r^{2}}\,dr=(\beta-\alpha)/2\pi.
\]
Hence the Gaussian measure of $E(\alpha,\beta)$ is $(\beta-\alpha)/2\pi$. Since there is an isometry $T$ of $\mathbb{R}^{n}$ such that $Tu=e_{1}$ and $Tv=(\cos\theta,\sin\theta,0,\dots,0)$, the left side of \eqref{2-eq1} may be expressed as
\[
\int_{\mathbb{R}^{n}}\sign\langle Tu,x\rangle \sign\langle x,Tv \rangle G_n^\mathbb{R}(x)\,dx.
\]
The set of $x$ where $\langle Tu,x\rangle>0$ and $\langle Tv,x\rangle>0$ is $E(\theta-\pi /2, \pi / 2)$, which has Gaussian measure $(\pi-\theta)/2\pi$; ditto for $\langle Tu,x\rangle < 0$ and $\langle Tv,x\rangle < 0$. The set of $x$ where $\langle Tu,x\rangle<0$ and $\langle Tv,x\rangle>0$ is $E(\pi / 2,\theta+ \pi / 2)$, which has Gaussian measure $\theta/2\pi$; ditto for $\langle Tu,x\rangle > 0$ and $\langle Tv,x\rangle < 0$. The set of $x$ where  $\langle Tu,x\rangle =0$ has zero Gaussian measure.
Hence the value of this integral is $ (\pi-\theta)/2\pi +(\pi-\theta)/2\pi  -\theta/2\pi  - \theta/2\pi =2\arcsin\langle u,v\rangle/\pi $.

\medskip

\noindent\textsc{Case} II: $\mathbb{F}=\mathbb{C}$.\; We define vectors $\alpha,\beta\in\mathbb{R}^{2n}$ with $\alpha_{2i-1}=\Re(u_i)$, $\alpha_{2i}=\Im(u_i)$, $\beta_{2i-1}=\Re(v_i)$, $\beta_{2i}=\Im(v_i)$,  $i=1,\dots,n$. Then $\alpha$ and $\beta$ are unit vectors in $\mathbb{R}^{2n}$.
For any $ z = (z_1,\dots,z_n)\in\mathbb{C}^n$, we write
\begin{gather*}
x=\bigl(\Re(z_1),\Im(z_1),\dots,\Re(z_n),\Im(z_n)\bigr)\in\mathbb{R}^{2n}.
\shortintertext{Then,}
\Re(\langle u,z\rangle)=\sum_{i=1}^n\Re(u_i\overline{z}_i)=\sum_{i=1}^n\bigl(\Re(u_i)\Re(z_i)+\Im(u_i)\Im(z_i)\bigr)
=\langle \alpha,x\rangle=\langle x,\alpha\rangle,
\end{gather*}
and likewise $\Re(\langle z,v\rangle) =\langle x,\beta\rangle$. 
By a change-of-variables and Case~I, we have
\begin{align}
\int_{\mathbb{C}^n}\sign\bigl(\Re\langle u,z\rangle\bigr) \sign\bigl(\Re\langle z,v\rangle\bigr) G_n^\mathbb{C}(z)\,dz
&= \int_{\mathbb{R}^{2n}}\sign\langle x,\alpha\rangle \sign\langle x,\beta\rangle G_{2n}^\mathbb{R}(x)\,dx \nonumber\\
&= \frac 2 \pi \arcsin \langle \alpha, \beta\rangle 
= \frac 2 \pi \arcsin(\Re\langle u,v\rangle).\label{2-eq3}
\end{align}
It is easy to verify that for any $z\in\mathbb{C}$, 
\begin{equation}\label{2-eq4}
\sign(z)=\frac 1 4\int_0^{2\pi}\sign(\Re(e^{-i\theta}z))e^{i\theta}\,d\theta.
\end{equation}
By  \eqref{2-eq3}, \eqref{2-eq4}, and Fubini's theorem,
\begin{align}
\int_{\mathbb{C}^n}&\sign\langle u,z\rangle \sign\langle z,v \rangle G_n^\mathbb{C}(z)\,dz \nonumber \\
&= \frac 1 {16}\int_0^{2\pi}\int_0^{2\pi}\int_{\mathbb{C}^n}\sign(\Re(\langle e^{-i\theta}u, z\rangle))
\sign(\Re(\langle z,\overline{e^{-i\varphi}}v\rangle))e^{i(\theta+\varphi)}G_n^\mathbb{C}(z)\,dz\,d\theta \,d\varphi  \nonumber  \\
&=\frac 1{8\pi}\int_0^{2\pi}\int_0^{2\pi}\arcsin ( \Re(\langle e^{-i\theta}u,e^{i\varphi}v\rangle))e^{i(\theta+\varphi)}\,d\theta \,d\varphi  \nonumber 
\intertext{\textsc{Case} II(a): $\langle u,v\rangle\in\mathbb{R}$.\; The integral above becomes}
&=\frac 1{8\pi}\int_0^{2\pi}\biggl[\int_0^{2\pi}\arcsin(\cos(\theta+\varphi)\langle u,v\rangle) e^{i(\theta+\varphi)}\,d\theta \biggr]\,d\varphi  \nonumber \\
&=\frac 1{8\pi}\int_0^{2\pi}\biggl[\int_{\varphi}^{2\pi+\varphi}\arcsin(\langle u,v\rangle\cos t) e^{it}\,dt \biggr]\,d\varphi  \nonumber \\
&=\frac 1{8\pi}\int_0^{2\pi}\biggl[\int_0^{2\pi}\arcsin(\langle u,v\rangle\cos t) e^{it}\,dt \biggr]\,d\varphi
=\frac 1 4\int_0^{2\pi}\arcsin(\langle u,v\rangle\cos t) e^{it}\,dt. \label{eq:last}
\end{align}
Since $\arcsin(\langle u,v\rangle\cos t)$ is an even function with period $2\pi$,
\[
\int_0^{2\pi}\arcsin(\langle u,v\rangle\cos t)\sin t \,dt=0,
\]
the last integral in \eqref{eq:last} becomes
\begin{gather}
\frac 1 4\int_0^{2\pi}\arcsin(\langle u,v\rangle\cos t)\cos t \,dt,  \nonumber
\intertext{and as $\arcsin(\langle u,v\rangle\cos t)\cos t$ is an even function with period $\pi$, it becomes}
\int_0^{\pi/2}\arcsin(\langle u,v\rangle\cos t)\cos t \,dt = \int_0^{\pi /2}\arcsin(\langle u,v\rangle\sin t)\sin t \,dt, \nonumber
\intertext{which, upon integrating by parts, becomes}
\langle u,v\rangle\int_0^{\pi /2}\frac{\cos^2t}{(1-|\langle u,v\rangle|^2\sin^2t)^{1/2}}\,dt. \label{2-eq5}
\end{gather}

\noindent \textsc{Case} II(b): $\langle u,v\rangle\notin\mathbb{R}$.\; This reduces to Case~II(a) by setting $c\in\mathbb{C}$ of unit modulus so that $c\langle u,v\rangle=|\langle u,v\rangle|$ and $\langle cu,v\rangle\in\mathbb{R}$, then by \eqref{2-eq5},
\begin{multline*}
\int_{\mathbb{C}^n}\sign\langle u,z\rangle \sign\langle z,v \rangle G_n^\mathbb{C}(z)\,dz
=\overline{c}\int_{\mathbb{C}^n}\sign\langle cu,z\rangle \sign\langle z,v \rangle G_n^\mathbb{C}(z)\,dz\\
=\overline{c}\langle cu,v\rangle\int_0^{\pi /2}\frac{\cos^2t}{(1-|\langle cu,v\rangle|^2\sin^2t)^{1/2}}\,dt
=\langle u,v\rangle\int_0^{ \pi/ 2}\frac{\cos^2t}{(1-|\langle u,v\rangle|^2\sin^2t)^{1/2}}\,dt.
\end{multline*}
\end{proof}

We will make a simple but useful observation\footnote{This of course follows from other well-known results but we would like to keep our exposition self-contained.} about the quantities in \eqref{GI} that we will need for the proof of Corollary~\ref{2-cor1} later.
\begin{lemma}\label{lem:maxepsdelta}
Let $\mathbb{F} = \mathbb{R}$ or $\mathbb{C}$ and $d,m,n\in\mathbb{N}$. For any $M\in\mathbb{F}^{m\times n}$, we have
\begin{gather}
\max_{\lvert \varepsilon_i \rvert \le 1, \; \lvert \delta_j \rvert \le 1}\biggl|\sum_{i=1}^m\sum_{j=1}^n M_{ij} \varepsilon_i \delta_j\biggr|=
\max_{\lvert \varepsilon_i \rvert = \lvert \delta_j \rvert = 1}\biggl|\sum_{i=1}^m\sum_{j=1}^n M_{ij} \varepsilon_i \delta_j\biggr| \label{eq:maxepsdelta}
\shortintertext{and for any $x_1,\dots,x_m, y_1,\dots,y_n\in\mathbb{F}^{d}$,}
\max_{\lVert x_i\rVert \le 1, \; \lVert y_j \rVert \le 1}\biggl| \sum_{i=1}^m\sum_{j=1}^n M_{ij} \langle x_{i}, y_{j}\rangle\biggr|=
\max_{\lVert x_i\rVert = \lVert y_j \rVert = 1}\biggl| \sum_{i=1}^m\sum_{j=1}^n M_{ij} \langle x_{i}, y_{j}\rangle\biggr|. \label{eq:maxxy}
\end{gather}
\end{lemma}
\begin{proof}
We will start with \eqref{eq:maxepsdelta}.
Suppose there exists $M\in \mathbb{F}^{m\times n}$ such that the left-hand side of \eqref{eq:maxepsdelta} exceeds the right-hand side.  Let the maximum of the left-hand side be attained by $\varepsilon_1^*,\dots,\varepsilon_m^*$ and  $\delta_1^*,\dots,\delta_n^*$.
By our assumption, at least one $\varepsilon_i^*$ or $\delta_j^*$ must be  less than $1$ in absolute value and so let $|\varepsilon_1^*|<1$ without loss of generality.  Fix $\varepsilon_i = \varepsilon_i^*$, $i = 2,\dots,m$ and $\delta_j = \delta_j^*$, $j=1,\dots,n$, but let $\varepsilon_1$ vary with $|\varepsilon_1|\le 1$ and consider the maximum of the left hand-side over $\varepsilon_1$.  Since $\max\{\lvert a\varepsilon_1+b\rvert : \lvert\varepsilon_1\rvert \le 1\}$ is always attained on the boundary $\lvert \varepsilon_1\vert=1$ for any $a,b\in \mathbb{F}$, this contradicts our assumption. The proof for \eqref{eq:maxxy} is similar with norm in place of absolute value.
\end{proof}

In the corollary below, the  inequality on the left is the ``original Grothendieck inequality,'' i.e., as first stated by Grothendieck\footnote{The better known modern version \eqref{GI} is in fact due to Lindenstrauss and Pe{\l}czy\'nski in \cite{Lindenstrauss}.} in \cite{Grothendieck}, and the inequality on the right is due to Haagerup \cite{Haagerup}. 
\begin{corollary}\label{2-cor1}
Let $\mathbb{F} = \mathbb{R}$ or $\mathbb{C}$ and $d,m,n\in\mathbb{N}$. For any $M\in\mathbb{F}^{m\times n}$ with
\begin{equation}\label{eq:inftyone}
\max_{\lvert \varepsilon_i \rvert = \lvert \delta_j \rvert = 1}\biggl|\sum_{i=1}^m\sum_{j=1}^n M_{ij} \varepsilon_i \delta_j\biggr|
\leq 1,
\end{equation}
any $x_1,\dots,x_m, y_1,\dots,y_n\in\mathbb{F}^{d}$ of unit $2$-norm, we have
\[
\biggl|\sum_{i=1}^m\sum_{j=1}^n M_{ij}\arcsin\langle x_{i}, y_{j}\rangle\biggr| \leq \frac \pi 2  \quad \text{if}\; \mathbb{F}=\mathbb{R},\qquad
\biggl|\sum_{i=1}^m\sum_{j=1}^n M_{ij}H(\langle x_{i}, y_{j}\rangle)\biggr| \leq 1 \quad \text{if}\; \mathbb{F}=\mathbb{C},
\]
where $H$ denotes the function on the right side of \eqref{2-eq1} for $\mathbb{F} = \mathbb{C}$.
\end{corollary}
\begin{proof}
The condition \eqref{eq:inftyone} implies that
\begin{align*}
\biggl|\sum_{i=1}^m\sum_{j=1}^nM_{ij}\sign\langle x_{i}, x\rangle \sign\langle y_{j},x\rangle G_d^\mathbb{R}(z)\biggr|\leq G_d^\mathbb{R}(z), \\ 
\biggl|\sum_{i=1}^m\sum_{j=1}^n M_{ij}\sign\langle z,\overline{x_i}\rangle \overline{\sign\langle z,\overline{y}_j \rangle}G_d^\mathbb{C}(z)\biggr|\leq G_d^\mathbb{C}(z),
\end{align*}
for any $x \in \mathbb{R}^{d}$, $z \in \mathbb{C}^{d}$ respectively.
Integrating over $\mathbb{R}^{d}$ or $\mathbb{C}^{d}$ respectively and applying Lemma~\ref{2-thm1} give the required results. Note that we have implicitly relied on \eqref{eq:maxepsdelta} in Lemma~\ref{lem:maxepsdelta} as the $\sign$ function  is not always of absolute value one and  may be zero.
\end{proof}

Corollary~\ref{2-cor1} already looks a lot like the Grothendieck inequaltiy \eqref{GI} but the nonlinear functions $\arcsin$ and $H$ are in the way. To obtain the  Grothendieck inequality, we linearize them: First by using Taylor series to replace these functions by polynomials; and then using a `tensor trick' to express the polynomials as linear functions on a larger space. This is the gist of the proof in Section~\ref{sec:proof}.

\section{Haagerup function}\label{sec:haa}

We will need to make a few observations regarding the functions on the right side of \eqref{2-eq1} for the proof of Grothendieck's inequality. Let the \emph{complex Haagerup function} of a complex variable $z$ be
\[
H(z) \coloneqq z\int_0^{\pi / 2}\frac {\cos^2t}{(1-|z|^2\sin^2t)^{1/2}}\,dt, \qquad \lvert z \rvert \le 1, 
\]
and the \emph{real Haagerup function} $h$ as the restriction of $H$ to $[-1,1] \subseteq \mathbb{R}$. Observe that $h:[-1,1]\to [-1,1]$ and is a strictly increasing continuous bijection. Since $[-1,1]$ is compact, $h$ is a homeomorphism of $[-1,1]$ onto itself. By the Taylor expansion
\begin{gather*}
(1-x^2\sin^2t)^{-1 /2}=\sum_{k=0}^{\infty}\frac{(2k-1)!!}{(2k)!!}x^{2k}\sin^{2k}t, \qquad |x|\leq1,\; 0\leq t< \pi/ 2,
\shortintertext{and}
\int_0^{\pi / 2}\cos^2t\sin^{2k}t\,dt=\frac{\pi}{4(k+1)}\frac{(2k-1)!!}{(2k)!!},
\end{gather*}
thus we get
\begin{equation}\label{h(x)}
h(x)=\sum_{k=0}^{\infty}\frac{\pi}{4(k+1)}\biggl[\frac{(2k-1)!!}{(2k)!!}\biggr]^2x^{2k+1}, \qquad x\in[-1,1].
\end{equation}

Since $h$ is analytic at $x=0$ and $h'(0)\neq 0$, its inverse function $h^{-1}: [-1,1] \to [-1,1]$ can be expanded in a power series in some neighborhood of $0$ 
\begin{equation}\label{inverse h}
h^{-1}(x)=\sum_{k=0}^{\infty}b_{2k+1}x^{2k+1}.
\end{equation}
One may in principle determine the coefficients using the Lagrange inversion formula:
\[
b_{2k+1}=\frac{1}{(2k+1)!}\lim_{t\to 0}\biggl[\frac{d^{2k}}{dt^{2k}}\biggl(\frac{t}{h(t)}\biggr)^{2k+1}\biggr].
\]
For example,
\[
b_1=\frac{4}{\pi}, \qquad b_3=-\frac{1}{8}\Bigl(\frac{4}{\pi}\Bigr)^3, \qquad b_5=0, \qquad b_7=-\frac{1}{1024}\Bigl(\frac{4}{\pi}\Bigr)^7.
\]
But determining  $b_{2k+1}$ explicitly becomes difficult as $k$ gets larger. A key step in Haagerup's proof \cite{Haagerup} requires the nonpositivity of the coefficients beyond the first:
\begin{equation}\label{eq:nonpos}
b_{2k+1}\leq0,\qquad \text{for all}\; k\geq1.
\end{equation}
This step  is in our view the most technical part of \cite{Haagerup}. We have no insights on how it may be avoided but we simplified Haagerup's proof of \eqref{eq:nonpos} in Section~\ref{non-positivity} to keep to our promise of an elementary proof --- using only calculus and basic complex variables.

It follows from \eqref{eq:nonpos} that  $\widetilde{h}(z) \coloneqq b_1 z -h^{-1}(z)$ has nonnnegative Taylor coefficients.   Pringsheim's theorem implies  that if the radius of convergence of the Taylor series of $\widetilde{h}(z)$ is $r$, then $\widetilde{h}(z)$, and thus $h^{-1}(z)$, has a singular point at $z = r$.  As $h'(t)>0$ on $(0,1)$ and $h(1)=1$, we must have $r\ge 1$.  It also follows from \eqref{eq:nonpos} that $h^{-1}(t)\le \sum_{k=0}^N b_{2k+1}t^{2k+1}$ for  any $t\in (0,1)$ and $N \in \mathbb{N}$.  So $\sum_{k=1}^N|b_{2k+1}|t^{2k+1}\le b_1t -h^{-1}(t)$ for any $t\in (0,1)$ and $N \in \mathbb{N}$. So $\sum_{k=1}^N|b_{2k+1}|\le b_1 -1$ for any $N \in \mathbb{N}$ and we have $\sum_{k=0}^{\infty}|b_{2k+1}|\le 2b_1 -1$.  As $h^{-1}(1)=1$ we deduce that $\sum_{k=0}^{\infty} b_{2k+1}=h^{-1}(1)=1$, and therefore  
\begin{equation}\label{eq:bound}
\sum_{k=0}^{\infty}|b_{2k+1}|= 2b_1 -1.
\end{equation}

We now turn our attention back to the complex Haagerup function. Observe that $|H(z)|=h(|z|)$ for all $z\in D \coloneqq \{z\in\mathbb{C}:|z|\leq1\}$  and $\arg(H(z))=\arg (z)$ for $0 \ne z\in D$. So $H:D\to D$ is a homeomorphism of $D$ onto itself. Let $H^{-1}:D\to D$ be its inverse function. Since $H(z)=\sign(z)h(|z|)$, we get
\begin{equation}\label{2-eq8}
H^{-1}(z)=\sign(z)h^{-1}(|z|)=\sign(z)\sum_{k=0}^{\infty}b_{2k+1}|z|^{2k+1}.
\end{equation}
Dini's theorem shows that the function $\varphi(x) \coloneqq \sum_{k=0}^{\infty}|b_{2k+1}|x^{2k+1}$ is a strictly increasing and continuous on $[0,1]$, with $\varphi(0)=0$ and $\varphi(1)=\sum_{k=0}^{\infty}|b_{2k+1}|\geq b_1= 4 / \pi>1$; note that $\varphi(1)$ is finite by \eqref{eq:bound}. Thus there exists a unique $c_0\in(0,1)$ such that $\varphi(c_0)=1$. So
\[
1=\varphi(c_0)=\sum_{k=0}^{\infty}|b_{2k+1}|c_0^{2k+1}=\frac 8 \pi c_0-h^{-1}(c_0),
\]
where the last equality follows from  $b_1= 4 /\pi$ and \eqref{eq:nonpos}.
Therefore we obtain $h^{-1}(c_0)=8 c_0/ \pi -1$, and if we let $x_0 \coloneqq h^{-1}(c_0)\in(0,1)$, then
$h(x_0)- \pi(x_0+1)/8=0$.
From the Taylor expansion of $h(x)$,  the function $x\mapsto  h(x)-\pi (x+1)/8$ is increasing and continuous on $[0,1]$. Hence $x_0$ is the unique solution in $[0,1]$  to
\begin{equation}\label{2-eq7}
h(x)-\frac \pi 8(x+1)=0
\end{equation}
and $c_0=\pi (x_0+1)/8$.

As Corollary~\ref{2-cor1} indicates, the Haagerup function $H$ plays the analogue of $\arcsin$ in the complex case. Unlike $\arcsin$, $H$ is a completely obscure function,\footnote{We are unaware of any other occurrence of $H$ outside its use in Haagerup's proof of his bound in \cite{Haagerup}.}  and any of its properties that we require will have to be established from scratch. The goal of this section is essentially to establish \eqref{h(x)}--\eqref{2-eq7}, which we will need later.

\section{A unified proof of Grothendieck's inequality}\label{sec:proof}

In this section we will need the notions of (i) tensor product and (ii)  Hilbert space, but just enough to make sense of $\mathcal{H}_n(\mathbb{F}) =\bigoplus_{k=0}^{\infty}(\mathbb{F}^n)^{\otimes(2k+1)}$ where $\mathbb{F} = \mathbb{R}$ or $\mathbb{C}$. In keeping to our promise of an elementary proof, we will briefly introduce these notions in a simple manner. For our purpose, it suffices to regard the \emph{tensor product} of $k$ copies of $\mathbb{F}^n$, denoted
\[
(\mathbb{F}^n)^{\otimes k} = \underbrace{\mathbb{F}^n \otimes \dots \otimes \mathbb{F}^n}_{k \text{ copies}},
\]
as the $\mathbb{F}$-vector space of $k$-dimensional \emph{hypermatrices},
\[
(\mathbb{F}^n)^{\otimes k} \coloneqq \bigl\{ [a_{i_1 \cdots i_k}] : a_{i_1 \cdots i_k} \in \mathbb{F}, \; i_1,\dots,i_k \in \{1,\dots,n\}\bigr\},
\]
where scalar multiplication and vector addition of hypermatrices are defined coordinatewise. For $k$ vectors $x, y, \dots, z \in \mathbb{F}^n$, their \emph{tensor product} is the $k$-dimensional hypermatrix given by
\[
x \otimes y \otimes \dots \otimes z \coloneqq [ x_{i_1} y_{i_2} \cdots z_{i_k} ]_{i_1, i_2, \dots, i_k =1}^n \in (\mathbb{F}^n)^{\otimes k}.
\]
We write
\[
 x^{\otimes k} \coloneqq \underbrace{x \otimes \dots \otimes x}_{k \text{ copies}}.
\]
If $\langle \cdot, \cdot \rangle$ is an inner product on $\mathbb{F}^n$, then defining
\begin{equation}\label{eq:ip1}
\langle x \otimes y \otimes \dots \otimes z, x' \otimes y' \otimes \dots \otimes z'\rangle \coloneqq \langle x,x'\rangle \langle y, y'\rangle \cdots \langle z, z'\rangle
\end{equation}
and extending bilinearly (if $\mathbb{F} = \mathbb{R}$) or sesquilinearly (if $\mathbb{F} =\mathbb{C}$) to all of $(\mathbb{F}^n)^{\otimes k} $ yields an inner product on the $k$-dimensional hypermatrices. In particular we have
\[
\langle x^{\otimes k}, y^{\otimes k} \rangle = \langle x, y\rangle^{k}.
\]
If $\{e_1, \dots, e_n \}$ is the standard orthonormal basis of  $\mathbb{F}^n$, then
\begin{equation}\label{eq:basis1}
 \bigl\{e_{i_1} \otimes \dots \otimes e_{i_k} \in (\mathbb{F}^{n})^{\otimes k} : i_1,\dots, i_k \in \{1,\dots,n\}\bigr\}
\end{equation}
is an orthonormal basis of $(\mathbb{F}^{n})^{\otimes k}$. For more information about hypermatrices see \cite{Lim} and for a more formal definition of tensor products see \cite{FA18}.

If an  $\mathbb{F}$-vector space $\mathcal{H}$  is equipped with an inner product $\langle \cdot, \cdot \rangle$ such that every Cauchy sequence in $\mathcal{H}$ converges with respect to the induced norm $\lVert v\rVert =\lvert \langle v, v\rangle \rvert^{1/2}$,  we call  $\mathcal{H}$ a \emph{Hilbert space}. Hilbert spaces need not be finite-dimensional; we call $\mathcal{H}$ \emph{separable} if there is a countable set of orthonormal vectors $\{e_j \in  \mathcal{H} : j \in J\}$, i.e., $J$ is a countable index set,  such that  every $v \in \mathcal{H}$ satisfies
\begin{equation}\label{eq:parseval}
\lVert v \rVert^2 = \sum_{j \in J} \lvert \langle v, e_j \rangle \rvert^2.
\end{equation}

Let $\langle \cdot, \cdot \rangle_k$ be the inner product on $(\mathbb{F}^n)^{\otimes(2k+1)}$ as defined in \eqref{eq:ip1}, $\lVert \, \cdot\, \rVert_k$ be its induced norm, and $B_k$ be the orthonormal basis in \eqref{eq:basis1}. Let $n \in \mathbb{N}$. The $\mathbb{F}$-vector space\footnote{The direct sum in \eqref{eq:Hn} is a Hilbert space direct sum, i.e., it is the closure of the vector space direct sum.}
\begin{equation}\label{eq:Hn}
\mathcal{H}_n(\mathbb{F}) \coloneqq \bigoplus_{k=0}^{\infty}(\mathbb{F}^n)^{\otimes(2k+1)}
= \Bigl\{(v_0, v_1, v_2, \dots) :  v_k \in (\mathbb{F}^n)^{\otimes(2k+1)}, \;  \sum\nolimits_{k=0}^\infty \lVert v_k \rVert_k^2 < \infty\Bigr\}
\end{equation}
equipped with the inner product
\begin{equation}\label{eq:ip2}
\langle u, v\rangle_* \coloneqq \sum_{k=0}^\infty \langle u_k, v_k \rangle_k
\end{equation}
is a separable Hilbert space since $\bigcup_{k=0}^{\infty} B_k$ is a countable set of orthonormal vectors satisfying  \eqref{eq:parseval}. We write  $\lVert \, \cdot \, \rVert_*$ for the norm induced by \eqref{eq:ip2}.

\begin{theorem}[Grothendieck inequality with Krivine and Haagerup bounds]\label{2-thm2}
Let $\mathbb{F} = \mathbb{R}$ or $\mathbb{C}$ and $l,m,n\in\mathbb{N}$. For any $M\in\mathbb{F}^{m\times n}$, any $x_1,\dots,x_m, y_1,\dots,y_n\in\mathbb{F}^{l}$ of unit $2$-norm, we have
\begin{equation}\label{eq:GKH}
\max_{\lVert x_i\rVert = \lVert y_j \rVert = 1}\biggl| \sum_{i=1}^m\sum_{j=1}^n M_{ij} \langle x_{i}, y_{j}\rangle\biggr|\leq K^\mathbb{F} \max_{\lvert \varepsilon_i \rvert = \lvert \delta_j \rvert = 1}\biggl|\sum_{i=1}^m\sum_{j=1}^n M_{ij} \varepsilon_i \delta_j\biggr|,
\end{equation}
where
\[
K^\mathbb{R} \coloneqq\frac \pi {2\log(1+\sqrt{2})} \qquad\text{and}\qquad
 K^\mathbb{C} \coloneqq \frac 8 {\pi(x_0+1)}
\]
are Krivine's and Haagerup's bounds respectively. Recall that $x_0$  is as defined in \eqref{2-eq7}.
\end{theorem}
\begin{proof}
As we described at the end of Section~\ref{sec:gauss}, we will `linearize' the nonlinear functions $\arcsin$ and $H$ in Corollary~\ref{2-cor1} by using Taylor series to replace these functions by polynomials, followed by a `tensor trick' to express polynomials as linear functions on an infinite-dimensional space.

\medskip

\noindent \textsc{Case I}: $\mathbb{F} = \mathbb{R}$.\; Let $c \coloneqq \text{arcsinh}(1)=\log(1+\sqrt{2})$. Taylor expansion gives
\begin{equation}
\sin(c\langle x_{i}, y_{j}\rangle)=\sum_{k=0}^{\infty}(-1)^{k}\frac{c^{2k+1}}{(2k+1)!}\langle x_{i}, y_{j}\rangle^{2k+1}
=\sum_{k=0}^{\infty}(-1)^{k}\frac{c^{2k+1}}{(2k+1)!}\bigl\langle x_{i}^{\otimes(2k+1)}, y_{j}^{\otimes(2k+1)}\bigr\rangle_k. \label{eq:tensor}
\end{equation}

For any $l \in \mathbb{N}$, let $\mathcal{H}_l(\mathbb{R})$ be as in \eqref{eq:Hn}, and $S,T:\mathbb{R}^l\to \mathcal{H}_l(\mathbb{R})$ be nonlinear maps defined by 
\[
\begin{aligned}
S(x)&\coloneqq \bigl(S_k(x)\bigr)_{k=0}^{\infty},\quad &S_k(x) &=(-1)^{k}\sqrt{\frac{c^{2k+1}}{(2k+1)!}}\cdot x^{\otimes(2k+1)},\\
T(x)&\coloneqq \bigl(T_k(x)\bigr)_{k=0}^{\infty},\quad &T_k(x) &=\sqrt{\frac{c^{2k+1}}{(2k+1)!}}\cdot x^{\otimes(2k+1)},
\end{aligned}
\]
for any $x\in\mathbb{R}^l$.
To justify that $S$ and $T$ are indeed maps into $\mathcal{H}_l(\mathbb{R})$, we need to demonstrate that $\|S(x)\|_*,\|T(x)\|_* <\infty$ but this follows from
\begin{gather*}
\|S(x)\|_*^2=\sum_{k=0}^{\infty}\|S_k(x)\|_k^2=\sum_{k=0}^\infty\frac{c^{2k+1}}{(2k+1)!}\|x\|^{2(2k+1)}=\sum_{k=0}^{\infty}\|T_k(x)\|_k^2 = \|T(x)\|_*^2
\shortintertext{and}
\sum_{k=0}^\infty\frac{c^{2k+1}}{(2k+1)!}\|x\|^{2(2k+1)} = \sinh(c\|x\|^2) < \infty
\end{gather*}
for all $x\in \mathbb{R}^l$.
Note that 
\[
\langle{S(x),T(y)\rangle_*=\sum_{k=0}^{\infty}(-1)^k\frac{c^{2k+1}}{(2k+1)!}}\langle x,y\rangle ^{2k+1}=\sin (c\langle x,y\rangle).
\]
Hence \eqref{eq:tensor} becomes:
\[
\sin(c\langle x_{i}, y_{j}\rangle)=\langle S(x_{i}), T(y_{j})\rangle_* \qquad\text{or}\qquad c\langle x_{i}, y_{j}\rangle=\arcsin\langle S(x_{i}), T(y_{j})\rangle_*.
\]
Moreover, since $x_{i}$ and $y_{j}$ are unit vectors in $\mathbb{R}^l$, we get
\[
\|S(x_{i})\|^{2}=\sinh(c\|x_{i}\|^{2})=1\qquad\text{and}\qquad \|T(y_{j})\|^{2}=\sinh(c\|y_{j}\|^{2})=1.
\]
As the $m+n$ vectors $S(x_1),\dots,S(x_m), T(y_1),\dots,T(y_n)$  in $\mathcal{H}_l(\mathbb{R})$ span a subspace $\mathcal{S} \subseteq \mathcal{H}_l(\mathbb{R})$ of dimension $d\le m+n$; and since any two finite-dimensional inner product spaces are isometric,  $\mathcal{S}$ is isometric to $\mathbb{R}^d$ with the standard inner product.
So we may apply Corollary~\ref{2-cor1} to obtain
\[
\biggl|\sum_{i=1}^m\sum_{j=1}^n M_{ij}\langle x_{i}, y_{j}\rangle\biggr|
= \frac{1}{c}  \biggl|\sum_{i=1}^m\sum_{j=1}^n M_{ij}\arcsin\langle S(x_{i}), T(y_{j})\rangle_*\biggr|
\leq \frac{\pi}{2c},
\]
which is Krivine's bound since $ \pi / 2c = \pi/ \bigl(2\log(1+\sqrt{2})\bigr) = K^\mathbb{R}$.

\medskip

\noindent\textsc{Case II}: $\mathbb{F} = \mathbb{C}$.\; Let $c_0 \in (0,1)$ be the unique constant defined in \eqref{2-eq7} such that $\varphi(c_0)=1$. By the Taylor expansion in \eqref{2-eq8} and noting that $\sign(z) \lvert z\rvert^{2k+1} = \overline{z}^{k} z^{k+1}$,
\begin{align}
H^{-1}(c_0\langle x_{i}, y_{j}\rangle)
&=\sign(c_0\langle x_{i}, y_{j}\rangle)\sum_{k=0}^{\infty}b_{2k+1}|c_0\langle x_{i}, y_{j}\rangle|^{2k+1} \nonumber \\
&= \sum_{k=0}^{\infty}b_{2k+1}c_0^{2k+1}  \overline{\langle x_{i}, y_{j}\rangle}^k \langle x_{i}, y_{j}\rangle^{k+1}\nonumber \\
&= \sum_{k=0}^{\infty}b_{2k+1} c_0^{2k+1}  \langle \overline{x}_{i}, \overline{y}_{j}\rangle^k \langle x_{i}, y_{j}\rangle^{k+1} \nonumber \\
&= \sum_{k=0}^{\infty}b_{2k+1} c_0^{2k+1} \bigl\langle \overline{x}_{i}^{\otimes k} \otimes x_{i}^{\otimes (k+1)}, \overline{y}_{j}^{\otimes k} \otimes  y_{j}^{\otimes (k+1)}\bigr\rangle_{k}.
 \label{eq:tensor2}   
\end{align}

For any $l \in \mathbb{N}$, let $D_l=\{x \in \mathbb{C}^l :\|x\|\le 1\}$ be the unit ball, let $\mathcal{H}_l(\mathbb{C})$ be as in \eqref{eq:Hn}, and let $S, T:D_l \to \mathcal{H}_l(\mathbb{C})$ be nonlinear maps defined by 
\[
\begin{aligned}
S(x)&= \bigl(S_k(x)\bigr)_{k=0}^{\infty},\quad &S_k(x) &\coloneqq  \sign(b_{2k+1})\bigl(|b_{2k+1}|c_0^{2k+1}\bigr)^{1/2} \cdot \bar x^{\otimes(k)}\otimes x^{\otimes(k+1)} ,\\
T(x)&= \bigl(T_k(x)\bigr)_{k=0}^{\infty},\quad &T_k(x) &\coloneqq \bigl(|b_{2k+1}|c_0^{2k+1}\bigr)^{1/2}\cdot \bar x^{\otimes(k)}\otimes x^{\otimes(k+1)},
\end{aligned}
\]
for any $x \in D_l$.
Then $S$ and $T$ are maps into $\mathcal{H}_l(\mathbb{C})$ since
\[
\|S(x)\|_*^2=\sum_{k=0}^{\infty}\|S_k(x)\|_k^2=\sum_{k=0}^\infty \lvert b_{2k+1} \rvert c_0^{2k+1}\|x\|^{2(2k+1)}=\sum_{k=0}^{\infty}\|T_k(x)\|_k^2 = \|T(x)\|_*^2
\]
and, as $b_1 > 0$ and $b_{2k+1} \le 0$ for all $k \ge 1$ by \eqref{eq:nonpos},
\[
\sum_{k=0}^\infty \lvert b_{2k+1} \rvert c_0^{2k+1}\|x\|^{2(2k+1)} 
=  2b_1  c_0 \|x\|^{2} - H^{-1}(c_0\|x\|^2) < \infty.
\]
As in Case I, we may rewrite \eqref{eq:tensor2} as
\[
H^{-1}(c_0\langle x_{i}, y_{j}\rangle)=\langle S(x_{i}), T(y_{j})\rangle_* \qquad\text{or}\qquad c_0\langle x_{i}, y_{j}\rangle=H(\langle S(x_{i}), T(y_{j})\rangle_*).
\]
Moreover, since $x_{i}$ and $y_{j}$ are unit vectors in $\mathbb{C}^l$, we get
\[
\|S(x_{i})\|^2=\sum_{k=0}^{\infty}|b_{2k+1}|c_0^{2k+1}=\varphi(c_0)=1,
\]
and similarly $\|T(y_j)\|=1$.
So we may apply Corollary~\ref{2-cor1} to get
\[
\biggl|\sum_{i=1}^m\sum_{j=1}^n M_{ij}\langle x_{i}, y_{j}\rangle\biggr|
=\frac 1{c_0}\biggl|\sum_{i=1}^m\sum_{j=1}^n M_{ij}H(\langle S(x_{i}), T(y_{j})\rangle_*)\biggr|
\leq\frac 1{c_0},
\]
which is Haagerup's bound since $1/c_0 = 8/\pi(x_0+1) = K^\mathbb{C}$.
\end{proof}

\section{Nonpositivity of $b_{2k+1}$}\label{non-positivity}

To make the proof in this article entirely self-contained, we present Haagerup's proof of the nonpositivity of $b_{2k+1}$ that we used earlier in \eqref{eq:nonpos}. While the main ideas are all due to Haagerup, our small contribution here is that we avoided the use of any known results of elliptic integrals  in order to stay faithful to our claim of an elementary proof, i.e., one that uses only calculus and basic complex variables.
To be clear, while the functions
\begin{equation}\label{eq:EK}
K(x)\coloneqq \int_0^{ \pi/ 2}(1-x^2\sin^2 t )^{-1 /2}\, dt, \qquad   E(x) \coloneqq \int_0^{ \pi / 2}(1-x^2\sin^2 t)^{ 1/ 2} \, dt
\end{equation}
do make a brief appearance in the proof of Lemma~\ref{basineqlem}, the reader does not need to know that they are  the complete elliptic integrals of the first and second kinds respectively.  Haagerup  had relied liberally on  properties of $K$ and $E$ that require substantial effort to establish  \cite{Haagerup}. We will only use trivialities that follow immediately  from definitions.

Our point of departure from Haagerup's proof is the following lemma about two functions $h_1$ and $h_2$, which we will  see in Lemma~\ref{extension} arise respectively from the real and imaginary parts of the analytic extension of the real Haagerup function $h:[-1,1]\to[-1,1]$  to the upper half plane.
\begin{lemma}\label{basineqlem}  
Let $h_1, h_2 : [1, \infty) \to \mathbb{R}$ be defined by
\begin{align*}
h_1(x)&\coloneqq\int_0^{\pi/2} \sqrt{1-x^{-2}\sin^2 t}\, dt,\\
h_2(x)&\coloneqq (1-x^{-2})\int_0^{\pi/2} \frac{\sin^2 t}{\sqrt{1-(1-x^{-2})\sin^2 t}} \, dt,
\end{align*} 
which are clearly strictly increasing functions on $[1,\infty)$ with
\[
h_1(1)= 1, \qquad \lim_{x\to \infty} h_1(x)=\pi/2, \qquad h_2(1)=0, \qquad \lim_{x\to \infty} h_2(x)=\infty.
\]
Then
\begin{alignat}{2}
\omega_1(x)\coloneqq x(h_1(x)h_2'(x) -h_1'(x)h_2(x)) &=\frac{\pi}{2}  &\qquad&\text{for}\;  x\ge 1,\label{thetincr}\\
\omega_2(x) \coloneqq x(h_1(x)h_1' (x)+h_2(x) h_2'(x))&\ge 2h_1(\sqrt{2})h_2(\sqrt{2}) >\frac{\pi}{4} &&\text{for}\;  1 \le x \le \sqrt{2}. \label{posomeg}
\end{alignat}
\end{lemma}
\begin{proof}
We start by observing some properties of $h_1'$ and $h_2'$. As
\[
h_1' (x)=\frac{1}{x^3}\int_0^{\pi/2}\frac{\sin^2 t}{\sqrt{1-x^{-2}\sin^2 t}}\, dt
=\frac{1}{x^2}\int_0^{\pi/2}\frac{\sin^2 t}{\sqrt{x^2-\sin^2 t}}\, dt,
\]
$h_1'$ is strictly decreasing on $(1,\infty)$. As $\int_0^{\pi/2} \cos^{-1} t \, dt=\infty$, $\lim_{x \to 1^+} h_1'(x)=\infty$. Clearly $\lim_{x\to \infty}h_1'(x)=0$.  Furthermore, when $x > 1$, since $\sqrt{x^2-\sin^2 t}\ge \sqrt{x^2-1}$, we have
\begin{equation}\label{eq:g1'bd}
0<h_1'(x)\le \frac{\pi}{4x^2\sqrt{x^2-1}} \textrm{ for } x>1.
\end{equation}
It is straightforward to see that the functions $E$ and $K$ in \eqref{eq:EK} have derivatives given by
\begin{equation}\label{elliptder}
  E'(y)=\frac{1}{y}\bigl(E(y)-K(y)\bigr), \qquad K'(y)=\frac{1}{y(1-y^2)}\bigl(E(y)-(1-y^2)K(y))\bigr).
\end{equation}
Clearly, $h_2(x)=K(y)-E(y)$, where $y=y(x)=\sqrt{1-x^{-2}}$. So by chain rule,
\[
h_2'(x)=y'(x)\frac{d}{dy}(K-E)(y(x)) =\frac{1}{x}\int_0^{\pi/2}(1-(1-x^{-2})\sin^2 t)^{1/2}\, dt.
\]
Hence $h_2'$ is strictly decreasing on $[1,\infty)$, $h_2'(1)=\pi/2$, and $\lim_{x\to \infty} h_2'(x)=0$.

To show \eqref{thetincr}, observe that 
\[
h_1(x)=E(1/x),\quad xh_1'(x)=K(1/x)-E(1/x),\quad h_2(x)=K(y)-E(y),\quad xh_2'(x)=E(y),
\]
where again $y=\sqrt{1-x^{-2}}$. Hence
\begin{align*}
\omega_1(x)&=E(1/x)E(y)-[K(1/x)-E(1/x)][K(y)-E(y)]\\
&=E(1/x)K(y)+K(1/x)E(y)-K(1/x)K(y).
\end{align*}
Computing $\omega_1'$, we see from \eqref{elliptder} that $\omega_1' \equiv 0$. So $\omega_1$ is  a constant function. By \eqref{eq:g1'bd}, $\lim_{x\searrow 1}h_1'(x)(1-x^{-2})=0$, and so $\lim_{x\searrow 1}\omega_1(x)=\pi/2$. Thus $\omega_1(x)=\pi/2$ for all $x > 1$ and we may set $\omega_1(1)=\pi/2$.

We now show \eqref{posomeg}  following  Haagerup's  arguments.
Note that
\[
\omega_2(x)=E(1/x)(K(1/x)-E(1/x))+E(y)(K(y)-E(y)).
\]
Let $g(x) \coloneqq E(\sqrt{x})(K(\sqrt{x})-E(\sqrt{x}))$.  A straightforward calculation using \eqref{elliptder} shows that
\[
g''(x)= \frac{1}{2}\biggl[ \frac{E(\sqrt{x})}{1-x}-\frac{K(\sqrt{x})-E(\sqrt{x})}{x}\biggr]^2\ge 0, \quad x\in[0,1].\]
So $g$ is convex on $[0,1]$.  Hence $g(1-x)$ is also convex on $[0,1]$.  Let $f(x) \coloneqq g(x)+g(1-x)$. Then $f$ is convex on $[0,1]$ and $f'(1/2)=0$.  Therefore $f(x)\ge f(1/2)\ge 2g(1/2)$.  This yields the first inequality in \eqref{posomeg}: $\omega_2(x)\ge 2h_1(\sqrt{2})h_2(\sqrt{2})$ for $x\in[1,\sqrt{2}]$.  

The Taylor expansions of $h_1$ and $h_2$ may be obtained as that in \eqref{h(x)},
\begin{align}
h_1(x)&=\frac \pi 2 \sum_{k=0}^\infty\biggl[\frac{(2k)!}{2^{2k}(k!)^2}\biggr]^2\frac{1}{1-2k}x^{-2k}, \label{eq:sf5-1}\\
h_2(x&)=\frac \pi 2 \sum_{k=0}^\infty\biggl[\frac{(2k)!}{2^{2k}(k!)^2}\biggr]^2\frac{2k}{2k-1}(1-x^{-2})^k. \label{eq:sf5-2}
\end{align}
Approximate numerical values of $h_1$ and $h_2$ at $x = \sqrt{2}$ and $4$ are calculated\footnote{For example, using  \url{http://www.wolframalpha.com}, which is freely available. Such numerical calculations cannot be completely avoided --- Haagerup's proof implicitly contains them as he used tabulated values of elliptic integrals.} to be:
\begin{equation}\label{eq:numval}
h_1(\sqrt{2}) \approx 1.3506438, \quad h_2(\sqrt{2}) \approx 0.5034307, \quad h_1(4) \approx 1.5459572, \quad h_2(4) \approx 1.7289033.
\end{equation}
The second inequality in \eqref{posomeg} then follows from
$2h_1(\sqrt{2})h_2(\sqrt{2}) \approx 2 \times 1.35064 \times 0.50343 >\pi/4$.
\end{proof}

In the next two lemmas and their proofs, $\Arg$ will denote principal argument.
\begin{lemma}\label{extension} Let $h:[-1,1]\to [-1,1]$ be the real Haagerup function as defined in Section~\ref{sec:haa}.  Then $h$  can be extended to a function $h_+: \overline{\mathbb{H}} \to \mathbb{C}$ that is continuous on the closed upper half-plane $\overline{\mathbb{H}}=\{z\in\mathbb{C} : \Im(z)\geq 0\}$ and analytic on the upper half-plane $\mathbb{H}=\{z\in\mathbb{C} : \Im(z)>0 \}$. In addition, $h_+$ has the following properties:
\begin{enumerate}[\upshape (i)]
  \item\label{extension-2} $\Im (h_+(z))\geq \Im(h_+(|z|))$ for all $z\in\overline{\mathbb{H}}\cap\{z\in\mathbb{C}: |z|\geq1 \}$ and $h_+(z)\ne 0$ for all  $z\in\overline{\mathbb{H}}\backslash \{0\}$. 
  \item \label{extension-3} For $x\in[1,\infty)$,
\[
\Re(h_+(x))= h_1(x), \qquad \Im(h_+(x))= h_2(x),
\]
where $h_1, h_2$ are as defined in Lemma~\ref{basineqlem}.

\item\label{extension-4} For all $k\in\mathbb{N}$ and all real $\alpha>1$,
\begin{equation}\label{eq:bkint}
b_{2k+1}=\frac{2}{\pi (2k+1)}\int_1^\alpha \Im\bigl( (h_+(x))^{-(2k+1)}\bigr)\,dx + r_k(\alpha)
\end{equation}
where 
\begin{equation}\label{eq:bkrem}
|r_k(\alpha)|\leq \frac{\alpha}{2k+1}\bigl(\Im (h_+(\alpha))\bigr)^{-(2k+1)}.
\end{equation}
\end{enumerate}
\end{lemma}
\begin{proof}
Integrating by parts, we obtain
\[
h(x)=\int_0^{\pi / 2}\cos t \cdot  d(\arcsin (x\sin t)) =\int_0^{\pi/2}\sin t\arcsin (x\sin t)\, dt, \qquad x\in[-1,1].
\]
The analytic function $\sin z$ is a bijection of $[-\pi/2, \pi/2]\times [0,\infty)$ onto $\overline{\mathbb{H}}$ and it maps the line segment $\{t+ia: -\pi/2\leq t\leq \pi/2\}$ onto the half ellipsoid $\{z\in\overline{\mathbb{H}}: |z-1|+|z+1|=2\cosh a   \}$.
Let $\arcsin_+$ be the inverse of this mapping. Then $\arcsin_+$ is continuous in $\overline{\mathbb{H}}$ and analytic in $\mathbb{H}$. In addition, we have:
\begin{align*}
\arcsin_+ x&=
\begin{cases}\arcsin x & \text{if}\;  x\in[-1,1],  \\
\frac \pi 2 \sign x +i\arccosh |x| & \text{if}\; x\in (-\infty,-1)\cup (1,\infty),
\end{cases}\\
\Im(\arcsin_+ z)&=\arccosh\Bigl(\frac 1 2 (|z-1|+|z+1|)\Bigr), \qquad z\in\overline{\mathbb{H}}.
\end{align*}
If we define
\[
h_+(z) \coloneqq \int_0^{\pi/2} \sin t \arcsin_+(z\sin t) \, dt,  \qquad   z\in\overline{\mathbb{H}},
\]
then $h_+$ is a continuous extension of $h$ to $\overline{\mathbb{H}}$ and is analytic in $\mathbb{H}$.

\begin{enumerate}[\upshape (i)]
\item Since $\arccosh$ is increasing on $[1,\infty)$, we have 
\[
\Im(\arcsin_+z)=\arccosh\Bigl(\frac 1 2 (|z-1|+|z+1|)\Bigr)\geq
\begin{cases}
\arccosh |z| & \text{if}\; |z|\geq 1, \\
0 & \text{if}\; |z|<1.
\end{cases}
\]
Therefore for $z\in\overline{\mathbb{H}}\cap\{z\in\mathbb{C}: |z|\geq1 \}$, 
\begin{align*}
\Im(h_+(z)) &= \int_0^{\pi/2} \sin t\cdot \Im(\arcsin_+(z\sin t))\, dt \\
&\geq \int_{\arcsin(1/|z|)}^{\pi/2} \sin t \arccosh (|z|\sin t) \, dt = \Im(h_+(|z|)).
\end{align*}
As $\Im(\arcsin_+ z)>0$ on $\mathbb{H}$, we have $\Im(h_+(z))>0$ on $\mathbb{H}$. For $x\in[-1,1]$, $h_+(x)=h(x)$ is zero only at $x=0$. For $x\in(-\infty,-1)\cup(1,\infty)$,
\[
\Im(h_+(x))=\int_{\arcsin(1/|x|)}^{\pi/2} \sin t \arccosh(|x|\sin t)\, dt>0.
\] 
Hence $h_+$ has no zero in $\overline{\mathbb{H}}\backslash \{0\}$.

\item Let $x\in(1,\infty)$. Integrating by parts followed by a change-of-variables $\sin u=x\sin t$ in the next-to-last equality gives us:
\begin{align*}
 \Re(h_+(x)) &= \int_0^{\arcsin(1/x)}\sin t \arcsin(x\sin t) \, dt +\frac{\pi}{2} \int_{\arcsin(1/x)}^{\pi/2} \sin t \, dt \\
&= x\int_0^{\arcsin(1/x)} \frac{\cos^2 t} {\sqrt{1-x^2\sin^2 t}} \, dt 
= \int_0^{\pi/2}\sqrt{1-x^{-2}\sin^2 u}\, du  
= h_1(x).
\end{align*}
A change-of-variables $\sin v=(1-x^{-2})^{-1/2}\cos t$ in the next-to-last equality gives us:
\begin{align*}
 \Im(h_+(x)) &=\int_{\arcsin(1/x)}^{\pi/2}\sin t \arccosh(x\sin t)  \, dt
= x\int_{\arcsin(1/x)}^{\pi/2} \frac{\cos^2 t}{\sqrt{x^2\sin^2 t-1}}  \, dt \\
&= (1-x^{-2})\int_0^{\pi/2} \frac{\sin^2 v}{\sqrt{1-(1-x^{-2})\sin^2 v}} \, dv =h_2(x).
\end{align*}

\item The power series \eqref{h(x)} shows that $h$ defines an analytic function $h(z)$ in the  open unit disk that is identically equal to $h_+(z)$ on $\{z\in\mathbb{C}: |z|<1\}\cap \overline{\mathbb{H}}$. Since $h(0)=0$ and $h'(0)\neq 0$, we can find some $\delta_0\in(0,1]$ such that $h(z)$ has an analytic inverse function \eqref{inverse h} in $\{z\in\mathbb{C}:|z|<\delta_0 \}$. For $0<\delta<\delta_0$, let $C_\delta$ be a counterclockwise orientated circle with radius $\delta$. It follows that $h(C_\delta)$ is a simple closed curve with winding number $+1$. Integrating by parts with a change-of-variables, we have 
\[
b_{2k+1} =\frac{1}{2\pi i}\int_{h(C_\delta)}\frac{h^{-1}(z)}{z^{2k+2}}\, dz 
= \frac{1}{2\pi i}\int_{C_\delta} \frac{z}{h(z)^{2k+2}}h'(z) \, dz.
\]
Note that $b_{2k+1}\in \mathbb{R}$ and 
\[
-(2k+1)\int_{C_\delta} \frac{z h'(z)}{h(z)^{2k+2}} \, dz+ \int_{C_\delta}\frac{1}{h(z)^{2k+1}}\, dz= \int_{C_\delta} \frac{d}{dz}\biggl[ \frac{z}{h(z)^{2k+1}}\biggr] \, dz=0.
\]
Then we get 
\begin{align*}
b_{2k+1}&=\frac{1}{2\pi (2k+1)}\int_{C_\delta} h(z)^{-(2k+1)}\,dz 
=\frac{1}{2\pi (2k+1)}\int_{C_\delta} \Im(h(z)^{-(2k+1)})\,dz\\
&= \frac{2}{\pi (2k+1)}\int_{C_\delta'} \Im(h(z)^{-(2k+1)})\,dz
\end{align*}
where $C_\delta' $ is the quarter circle $\{ \delta e^{i\theta}: 0\leq \theta \leq \pi/2  \}$. Since $h(z)$ identically equals $h_+(z)$ on $C_\delta'$ and $h_+(z)$ has no zeros in the set $\{z\in\mathbb{C}: \delta\leq |z|\leq \alpha,\; 0\leq \Arg z\leq \pi/2 \}$ by \eqref{extension-2},  Cauchy's integral formula yields
\[
b_{2k+1}= \frac{2}{\pi (2k+1)}\Im\biggl[ \int_\delta^\alpha h_+(z)^{-(2k+1)}\,dz + \int_{C_\alpha'} h_+(z)^{-(2k+1)}\,dz  + \int_{i\alpha}^{i\delta} h_+(z)^{-(2k+1)}\,dz   \biggr].
\]
Moreover, since $h_+(z)$ is real on $[\delta, 1]$ and its real part vanishes on the imaginary axis, we are left with
\[
b_{2k+1}= \frac{2}{\pi (2k+1)} \int_1^\alpha \Im( h_+(z)^{-(2k+1)}) \,dz+ \frac{2}{\pi (2k+1)}\Im \biggl[ \int_{C_\alpha'} h_+(z)^{-(2k+1)}\,dz \biggr].
\]
By \eqref{extension-2}, $h_+(z)\geq  \Im(h_+(z))\geq  \Im(h_+(|z|))$. Thus 
\[
\biggl| \int_{C_\alpha'} h_+(z)^{-(2k+1)}\,dz  \biggr|\leq \frac{\pi\alpha}{2}\bigl(\Im(h_+(\alpha))\bigr)^{-(2k+1)}.  \qedhere
\]
\end{enumerate}
\end{proof}

The integral expression of $b_{2k+1}$ in \eqref{eq:bkint} will be an important ingredient in the proof that $b_{2k+1}\leq 0$ for $k\geq 1$. We establish some further approximations for this integral in the next and final lemma. 
\begin{lemma}\label{approximations}
Let $\alpha=4$ throughout.\footnote{To avoid confusion, we write `$\alpha$' for the upper limit of our integrals instead of `$4$' as the same number will also appear in an unrelated context `$k \ge 4$.'}
Let $\theta(x) \coloneqq \Arg(h_+(x))$ for $x\in[1,\infty)$. Then $\theta :[1,\infty) \to [0, 2\pi]$ is strictly increasing on for $x\ge 1$, $\theta(1)=0$, and $\lim_{x\rightarrow\infty}\theta(x)=\pi/2$. In addition, we have the following:
\begin{enumerate}[\upshape (i)]
\item\label{approx-2} Let  $p\coloneqq \lfloor (2k+1)\theta(\alpha)/\pi\rfloor$. Let
\begin{align*}
I_r &\coloneqq \frac{2}{\pi(2k+1)}\int_{\theta(x)=\pi(r-1)/(2k+1)}^{\theta(x)=\pi r/(2k+1)} |h_+(x)|^{-(2k+1)}|\sin( (2k+1)\theta(x))|\, dx
\intertext{for $r = 1,2,\dots,p$, and}
J &\coloneqq \frac{2}{\pi(2k+1)}\int_{\theta(x)=\pi p/(2k+1)}^{\alpha} |h_+(x)|^{-(2k+1)}|\sin((2k+1)\theta(x))| \, dx.
\end{align*}
Then 
\[
\frac {2} {\pi (2k+1)}\int_1^\alpha \Im(h_+(x)^{-(2k+1)}) \, dx = -I_1 + I_2- \ldots +(-1)^p I_p + (-1)^{p+1}J.
\]
\item\label{approx-3} Let $k\geq 4$. Then $p\geq 2$ and $I_1>I_2>\dots>I_p>J$.
\item\label{approx-4} Let $k\geq 4$  and $c=|h_+(\sqrt{2})|e^{-\theta(\sqrt{2})/2}$. Then $I_1>0.57c^{-(2k+1)}/(2k+1)^2$ and $I_2<0.85 I_1$.
\end{enumerate}
\end{lemma}
\begin{proof}
Since $\theta(x) = \arctan (h_2(x)/h_1(x))$, by \eqref{thetincr},  we get 
\begin{equation}\label{theta_derivative}
\frac{d\theta(x)}{dx}=\frac{h_1(x)h_2'(x)-h_1'(x)h_2(x)}{|h_+(x)|^2}>0, \qquad x>1.
\end{equation}
So $\theta(x)$ is strictly increasing on for $x\ge 1$. It is clear that $\theta(1)=0$. By Lemma~\ref{basineqlem},
$\lim_{x\rightarrow \infty} h_1(x)= \pi /2$ and $\lim_{x\rightarrow \infty} h_2(x)=+\infty$, so $\lim_{x\rightarrow\infty}\theta(x)=\pi/2$.

\begin{enumerate}[\upshape (i)]
\item This follows from dividing the interval of the integral $[1,\alpha]$ into $p+1$ subsets:
\begin{align*}
\frac {2} {\pi (2k+1)}\int_1^\alpha \Im(h_+(x)^{-(2k+1)}) \, dx  &= -\frac{2}{\pi(2k+1)}\int_1^\alpha |h_+(x)|^{-(2k+1)}\sin( (2k+1)\theta(x))\, dx \\
&= -I_1 + I_2- \ldots +(-1)^p I_p + (-1)^{p+1}J.
\end{align*}

\item We write $x=x(\theta)$, $\theta\in[0,\pi/2)$, for the inverse function of $\theta=\theta(x)$. By \eqref{theta_derivative}, we have 
\begin{align*}
I_r&=\frac{4}{\pi^2(2k+1)}\int_{\pi(r-1)/(2k+1)}^{\pi r/(2k+1)} x(\theta) |h_+(x(\theta))|^{-2k+1}|\sin ((2k+1)\theta)| \, d\theta, \\
J&=\frac{4}{\pi^2(2k+1)}\int_{\pi p/(2k+1)}^{\theta(\alpha)} x(\theta) |h_+(x(\theta))|^{-2k+1}|\sin ((2k+1)\theta)| \, d\theta.
\end{align*}

By Lemma~\ref{basineqlem}, $h_1(x)$ and $h_2(x)$ are strictly increasing function of $x\in [1,\infty)$, therefore, so is $|h_+(x)|^2=h_1(x)^2 + h_2(x)^2$.  With this, we deduce that $x|h_+(x)|^{-2k+1}$ is strictly decreasing on $[1,\alpha]$ for $k\geq 4$ as
\begin{align*}
\frac{d}{dx}(x|h_+(x)|^{-2k+1}) &=|h_+(x)|^{-2k+1}+\frac{(-2k+1)x}{2}|h_+(x)|^{-2k-1}\frac{d}{dx}|h_+(x)|^2 \\
&=|h_+(x)|^{-2k-1}\bigl(|h_+(x)|^2-(2k-1)x(h_1(x)h_1'(x)+h_2(x)h_2'(x))\bigr) \\
&\leq |h_+(x)|^{-2k-1}\Bigl(|h_+(x)|^2-(2k-1)\frac{\pi}{4} \Bigr) \\
&\leq |h_+(x)|^{-2k-1}\Bigl(|h_+(\alpha)|^2-\frac{7\pi}{4} \Bigr) \approx -0.1187<0,
\end{align*}
where we have used the fact that $|h_+(x)|^2$ is increasing on $[1,\alpha]$ in the next-to-last inequality and the numerical value is calculated from those of $h_1(4)$ and $h_2(4)$ in \eqref{eq:numval}. Since $|\sin((2k+1)\theta)|$ is periodic with period $\pi/(2k+1)$, we obtain $I_1>I_2>\ldots>I_p$. In addition, 
\begin{align*}
J &=\frac{4}{\pi^2(2k+1)}\int_{\pi p/(2k+1)}^{\theta(\alpha)} x(\theta) |h_+(x(\theta))|^{-2k+1}|\sin ((2k+1)\theta)| \, d\theta \\
&\leq \frac{4}{\pi^2(2k+1)}\int_{(p-1)\pi/(2k+1)}^{\theta(\alpha)-\pi/(2k+1)} x(\theta) |h_+(x(\theta))|^{-2k+1}|\sin ((2k+1)\theta)| \, d\theta < I_p.
\end{align*}
Finally, we have $\theta(\alpha)=\arctan(h_1(\alpha)/h_2(\alpha))\approx 0.8412>\pi/4 =\arctan(1)$, and so $p=\lfloor (2k+1)\theta(\alpha)/\pi\rfloor \geq \lfloor 9\theta(\alpha)/\pi\rfloor =2$ for $k\geq 4$.

\item Since $x(\theta)\geq 1$ for $\theta\in[0,\pi/2)$, we have 
\[
I_1\geq \frac{4}{\pi^2(2k+1)}\int_0^{\pi /(2k+1)} x(\theta) |h_+(x(\theta))|^{-2k+1}|\sin ((2k+1)\theta)| \, d\theta.
\]
Recall that $\theta=\theta(x)$ and $x=x(\theta)$ are inverse functions of one another. For $\theta\in[0,\theta(\sqrt{2})]$,
\begin{align*}
\frac{d}{d\theta}\log|h_+(x(\theta))|  &=\frac{1}{2} \frac{d}{dx}\log|h_+(x)|^2\cdot \Bigl(\frac{d\theta}{dx}  \Bigr)^{-1}  \\
&= \frac{h_1(x)h_1'(x)+h_2(x)h_2^\prime(x)}{ h_1(x)h_2^\prime(x)-h_1^\prime(x)h_2(x)}= \frac{\omega_2(x)}{\omega_1(x)}>\frac 1 2,
\end{align*}
for $x\in [1,\sqrt{2}]$, where we have used \eqref{thetincr}, \eqref{posomeg}, and the fact that $\theta(x)$ is strictly increasing for $x\ge 1$.
Hence $\log|h_+(x(\theta))|\leq \log|h_+(\sqrt{2})|-(\theta(\sqrt{2})-\theta)/2$ which is equivalent to 
\[
|h_+(x(\theta))|\leq ce^{\theta/2}, \qquad \theta\in[0,\theta(\sqrt{2})]
\]
where $c=|h_+(\sqrt{2})|e^{-\theta(\sqrt{2})/2}\approx 1.2059$ and $\theta(\sqrt{2})>\pi/9$, using values of  $h_1(\sqrt{2})$ and $h_2(\sqrt{2})$ in \eqref{eq:numval}. 

It follows that for $k\geq 4$, we have 
\begin{align*}
I_1 &\geq \frac{4}{\pi^2(2k+1)}\int_0^{\pi /(2k+1)} (ce^{\theta/2})^{-2k+1}\sin ((2k+1)\theta) \, d\theta  \\
&= \frac{4c^{-2k+1}}{\pi^2(2k+1)^2}\int_0^{\pi} e^{-(k-1/2)\theta/(2k+1)}\sin \theta \, d\theta 
\geq  \frac{4c^{-2k+1}}{\pi^2(2k+1)^2}\int_0^{\pi} e^{-\theta/2}\sin \theta \, d\theta \\
&= \Bigl(\frac{2c}{\pi} \Bigr)^2\frac{1+e^{-\pi/2}}{1+1/4}\cdot \frac{c^{-(2k+1)}}{(2k+1)^2} > \frac{0.57c^{-(2k+1)}}{(2k+1)^2}.
\end{align*}
Since $\frac{d}{d\theta}\log|h_+(x(\theta))|\geq 1/2$, we get 
\[
|h_+(x(\theta+\pi/(2k+1)))|^{-2k+1}\leq e^{-(k-1/2)\pi/(2k+1)}|h_+(x(\theta))|^{-2k+1}
\]
Moreover, since $\theta(5/\sqrt{3})>2\pi/9$, we know that $x(\theta)\leq 5/\sqrt{3}$ on $[0,2\pi/9]$. Hence for $k\geq 4$, it follows from the above results that
\begin{align*}
I_r &=\frac{4}{\pi^2(2k+1)}\frac{5}{\sqrt{3}}\int_{\pi/(2k+1)}^{2\pi/(2k+1)} x(\theta) |h_+(x(\theta))|^{-2k+1}|\sin ((2k+1)\theta)| \, d\theta  \\
&=\frac{4}{\pi^2(2k+1)}\frac{5}{\sqrt{3}}\int_0^{\pi/(2k+1)} x(\theta) |h_+(x(\theta+\pi/(2k+1)))|^{-2k+1}\sin ((2k+1)\theta) \, d\theta   \\
&\leq \frac{4}{\pi^2(2k+1)}\frac{5}{\sqrt{3}}e^{-(k-1/2)\pi /(2k+1)}\int_0^{\pi/(2k+1)} x(\theta) |h_+(x(\theta))|^{-2k+1}\sin ((2k+1)\theta) \, d\theta   \\
&\leq \frac{5}{\sqrt{3}}e^{-7\pi /18} I_1 < 0.85I_1. \qedhere
\end{align*}
\end{enumerate}
\end{proof}

The fact that $x|h_+(x)|^{-2k+1}$ is strictly decreasing on $[1,4]$ for $k\geq 4$, established in the proof of \eqref{approx-3} above, is a crucial observation for establishing the nonpositivity of $b_{2k+1}$ for $k\ge 4$.  Observe that since $|h_+(x)|$ is strictly increasing for $x>1$, it is enough to show that $x|h_+(x)|^{-7}$ is strictly decreasing on $[1,4]$, which is what we did.  Note that for a fixed $k\ge 1$, $x|h_+(x)|^{-2k+1}$ is \emph{increasing} for large enough $x$, as $|h_+(x)|$ behaves like $C\log x$ for $x \gg 1$.

\begin{theorem}\label{thm:nonpos}
Let the Taylor expansion of $h^{-1}(x)$ be as in \eqref{inverse h}.  Then $b_{2k+1}\leq 0$ for $k\geq1$. 
\end{theorem}
\begin{proof}
Let $k\geq 4$ and let $I_1, I_2,\ldots, I_p, J $ be as defined in Lemma~\ref{approximations}. By \eqref{eq:bkint} with $\alpha=4$ and Lemma~\ref{approximations}\eqref{approx-2} and \eqref{approx-3}, we have 
\[
-b_{2k+1}  =I_1-I_2+\ldots+(-1)^{p-1}I_p+(-1)^p J-r_{2k+1}(5\sqrt{2}) > I_1-I_2-r_{2k+1}(5\sqrt{2}).
\]
By \eqref{eq:bkrem} and Lemma~\ref{approximations}\eqref{approx-4} with $c\approx 1.2059$ (established in its proof), we get 
\[
I_1-I_2>\frac{0.0855}{(2k+1)^2}(1.206)^{-(2k+1)}, \qquad |r_{2k+1}(4)|\leq \frac{4}{2k+1}(1.728)^{-(2k+1)}.
\]
Since $-b_{2k+1}> I_1-I_2-r_{2k+1}(4)$, we get $b_{2k+1}<0$ for $k\geq 9$. Direct computation using the Lagrange inversion formula gives us $b_3, b_5, \dots, b_{17} \leq 0$, proving  nonpositivity  for $k \le 8$.
\end{proof}

\section*{Acknowledgment}

We thank the anonymous referee for his careful reading and many helpful suggestions that greatly improved our article.
The work in this article is generously supported by DARPA D15AP00109 and NSF IIS 1546413. LHL acknowledges additional support from a DARPA Director's Fellowship and the Eckhardt Faculty Fund.

\bibliographystyle{abbrv}

\end{document}